\documentclass[british,english]{article}
\usepackage[T1]{fontenc}
\usepackage[latin9]{inputenc}
\usepackage{verbatim}
\usepackage{float}
\usepackage{amsmath}
\usepackage{amsthm}
\usepackage{amssymb}
\usepackage{graphicx}
\usepackage{esint}

\makeatletter
\theoremstyle{plain}
\newtheorem{thm}{\protect\theoremname}
  \theoremstyle{plain}
  \newtheorem{cor}[thm]{\protect\corollaryname}

\@ifundefined{date}{}{\date{}}
\AtBeginDocument{
\let\origtableofcontents=\tableofcontents
\def\tableofcontents{\@ifnextchar[{\origtableofcontents}{\gobbletableofcontents}}
\def\gobbletableofcontents#1{\origtableofcontents}
}

\usepackage{babel}
\addto\captionsbritish{\renewcommand{\corollaryname}{Corollary}}
  \addto\captionsbritish{\renewcommand{\theoremname}{Theorem}}
  \addto\captionsenglish{\renewcommand{\corollaryname}{Corollary}}
  \addto\captionsenglish{\renewcommand{\theoremname}{Theorem}}
  \providecommand{\corollaryname}{Corollary}
\providecommand{\theoremname}{Theorem}

\usepackage{babel}
\addto\captionsbritish{\renewcommand{\corollaryname}{Corollary}}
  \addto\captionsbritish{\renewcommand{\theoremname}{Theorem}}
  \addto\captionsenglish{\renewcommand{\corollaryname}{Corollary}}
  \addto\captionsenglish{\renewcommand{\theoremname}{Theorem}}
  \providecommand{\corollaryname}{Corollary}
\providecommand{\theoremname}{Theorem}

\usepackage{babel}
\addto\captionsbritish{\renewcommand{\corollaryname}{Corollary}}
  \addto\captionsbritish{\renewcommand{\theoremname}{Theorem}}
  \addto\captionsenglish{\renewcommand{\corollaryname}{Corollary}}
  \addto\captionsenglish{\renewcommand{\theoremname}{Theorem}}
  \providecommand{\corollaryname}{Corollary}
\providecommand{\theoremname}{Theorem}

\makeatother

\usepackage{babel}
  \addto\captionsbritish{\renewcommand{\corollaryname}{Corollary}}
  \addto\captionsbritish{\renewcommand{\theoremname}{Theorem}}
  \addto\captionsenglish{\renewcommand{\corollaryname}{Corollary}}
  \addto\captionsenglish{\renewcommand{\theoremname}{Theorem}}
  \providecommand{\corollaryname}{Corollary}
\providecommand{\theoremname}{Theorem}

\begin{document}

\title{{Probabilistic approach to a cell growth model } }

\author{{\normalsize{}{}{}{}Gregory Derfel}{\small{}{}{}{}}\\
 {\small{}{}{}{}Department of Mathematics}\\
 {\small{}{}{}{}\vspace{0.3cm}
 Ben-Gurion University of the Negev Beersheba, Israel}\\
 {\normalsize{}{}{}{}Yaqin Feng}{\small{}{}{}{}}\\
 {\small{}{}{}{}Department of Mathematics }\\
 {\small{}{}{}{}\vspace{0.3cm}
 Ohio University, Athens, Ohio 45701,USA}\\
 {\normalsize{}{}{}{}Stanislav Molchanov }{\small{}{}{}{}}\thanks{This author was partially supported by Russian Science Foundation,
Project No 17-11-01098. }\\
 {\small{}{}{}{}Department of Mathematics and Statistics }\\
 {\small{}{}{}{}University of North Carolina at Charlotte, Charlotte,
NC 28223,USA }\\
 {\small{}{}{}{}\vspace{0.3cm}
 National Research University, Higher School of Economics, Russian
Federation }\\
 }

\maketitle
\textbf{Abstract:} We consider the time evolution of the supercritical
Galton-Watson model of branching particles with extra parameter (mass).
In the moment of the division the mass of the particle (which is growing
linearly after the birth) is divided in random proportion between
two offsprings (mitosis). Using the technique of moment equations
we study asymptotics of the the mass-space distribution of the particles.
Mass distribution of the particles is the solution of the 
equation with linearly transformed argument: functional, functional-differential
or integral. We derive several limit theorems describing the fluctuations
of the density of the particles, first two moments of the total masses
etc. Also, we consider the branching process in the presence of a
random spatial motion (say, diffusion). Here we discuss the classical
Fisher, Kolmogorov, Petrovski, Piskunov model and distribution of
mass inside the propogation front.

\section{Introduction}

A model for the simultaneous growth and division of a cell population,
structured by size, was introduced and studied by Hall and Wake \cite{Hall1989ASFCS}
(cf. \cite{Round1990ASFCS}). The original model deals with symmetrical
cell-division, where each cell divides into $\kappa$ equally sized
daughter cells. Under this assumption Hall and Wake proved that the
steady-size mass distribution exists and satisfies the celebrated
pantograph functional-differential equation 
\begin{equation}
y'(x)=ay(\kappa x)+by(x)\label{eq:pantograph}
\end{equation}
where $\kappa>1$. Since then, different variations and extensions
of the original model have been studied and used to describe plant
cells, diatoms (\cite{Bass2004ASFCS}, \cite{Begg2008ASFCS}, \cite{Derfel2009ASFCS},
\cite{Daukste_Basse_Baguley_Wall_2012}) and also tumor growth \cite{BBMJBW_tumor2003}
.

In the present paper, in order to describe cell growth model 
we use the supercritical Galton-Watson model of branching particles
with extra parameter (mass). Similar approach was applied earlier
in \cite{Derfel2009ASFCS}.

Namely, we assume that the mass of the particle is growing linearly
between the exponentially distributed splitting moments and that in
the moment of the division the mass of the particle is divided in
a random proportion between two offspring (see Figure 1 below). Notice
that under these assumptions splitting moments depend on mass.




The model described above gives only rather schematic description
of the cell growth process, but it is interesting from the mathematical
point of view, and hopefully in some cases may reflect an important
qualitative features of real biological systems (\cite{Bass2004ASFCS},
\cite{Begg2008ASFCS}, \cite{Derfel2009ASFCS}, \cite{Daukste_Basse_Baguley_Wall_2012},
\cite{BBMJBW_tumor2003}).

We start from the study of the total number of the particles $N(t)$,
their distribution with respect to the mass, and first two moments
of the total mass distribution. We study the asymptotic of the mass
distribution of the particles and prove several limit theorems.

Let $N(t)$ be the supercritical Galton-Watson process \cite{gw1975}
with mortality rate $\mu$ and splitting rate $\beta>\mu\geq0$. Assume
that at initial time 0, $N(0)=1$. For the generating function $u_{z}(t)=Ez^{N(t)},$
we have the well known equation \cite{teh}: 
\begin{eqnarray}
\frac{\partial u_{z}(t)}{\partial t} & = & \beta u_{z}(t)^{2}-(\beta+\mu)u_{z}(t)+\mu\label{eq: generating function}\\
u_{z}(0) & = & 1.\nonumber 
\end{eqnarray}
Let $\delta=\beta-\mu$ and $\gamma=\frac{\mu}{\beta}$, elementary
calculations give for $N(t)$ the geometric distribution 
\begin{equation}
P(N(t)=k)=\frac{\left(1-\gamma\right)^{2}\left(1-e^{\delta t}\right)^{k-1}e^{\delta t}}{\left(\gamma-e^{\delta t}\right)^{k+1}}\,\,,k\geq1,\label{Chapequation4}
\end{equation}

\begin{equation}
P(N(t)=0)=\frac{1-e^{\delta t}}{1-\frac{1}{\gamma}e^{\delta t}},\,\,\label{Chapequation4-1}
\end{equation}

\begin{eqnarray*}
EN(t)=e^{\delta t},
\end{eqnarray*}
and for any $a\geq0$, as $t\rightarrow\infty,$ 
\begin{eqnarray*}
P\left\{ \frac{N(t)}{e^{\delta t}}\in(a,a+da)\right\} =(1-\gamma)e^{-a}+\gamma\delta_{0}(a).
\end{eqnarray*}
Assume now that the initial particles have mass $m>0$. The evolution
of the mass $m(t)$ includes two features. First of all we assume
the linear growth: 
\[
m(t)=m+vt,
\]
where $v>0$ until the splitting. Probability of the splitting in
each time interval $[t,t+dt]$ equals $\beta dt$, i.e. the moment
$\tau_{1}$ of splitting of the initial particle has exponential law
with parameter $\beta$ and $P\left\{ \tau>t\right\} =e^{-\beta t}$.
At the moment $\tau_{1}+0$, the particle with mass $m+v\tau_{1}$
is divided into two particles with the random masses:$m'=\theta(m+v\tau_{1})$,
$m''=(1-\theta)(m+v\tau_{1})$. Here $\theta\in[0,1]$ is symmetrically
distributed (with respect to the center $0.5\in[0,1]$) random variable
which has the density $q(x)=q(1-x)$, $x\in[0,1]$. As usually we
assume that the random variable $\theta_{i}$, $i=1,2,\cdots$ for
different splittings are independent. The dynamics of the sub-populations
generated by different offspring are also independent.

Let us introduce the main object of our study: the moment generating
function of the two random variables: $N(t):=$ total numbers of particles
at the moment $t>0,$ $M(t):=$ total mass of the particles at the
moment $t$. We introduce the generating function of the form: 
\begin{equation}
u(t,m;z,k)=E_{m}z^{N(t)}e^{-kM(t)},|z|\leq1,k\geq0.\label{eq:momentgenerating}
\end{equation}
The following results are the basis for the further analysis: 
\begin{thm}
\label{Thereom moment generating}Let $u(t,m;z,k)=Ez^{N(t)}e^{-kM(t)},$
then $u(t,m;z,k)$ satisfy the following functional-differential equation:

\begin{eqnarray}
\frac{\partial u(t,m;z,k)}{\partial t} & = & \frac{\partial u(t,m;z,k)}{\partial m}v+\beta\int_{0}^{1}u(t,\theta m;z,k)\cdot u(t,(1-\theta)m;z,k)q(\theta)d\theta\nonumber \\
 &  & -(\beta+\mu)u(t,m;z,k)+\mu\label{eq:basicequations}\\
u(0,m;z,k) & = & ze^{-km}.\nonumber 
\end{eqnarray}
\end{thm}

\begin{proof}
The formal derivation of this equation is based on the standard technique:
balance of the probabilities in the infinitesimal initial time interval
$[0,dt]$. Namely, let's consider 
\[
u(t+dt,m;z,k)=E_{m}z^{N(t+dt)}e^{-kM(t+dt)}
\]
and then let's split the interval $[0,t+dt]$ into two parts $[0,dt]\cup[dt,t+dt]$.
At the moment $t=0$, we have one particle in the point $x$ with
mass $m$ and during $[0,dt]$, we observe one of the following : 
\begin{itemize}
\item splitting of the initial particle into two particles with probability
$\beta dt$; 
\item annihilation of the initial particle with probability $\mu dt$; 
\item nothing happen, no annihilation and no splitting with probability
$1-\beta dt-\mu dt.$ 
\end{itemize}
Now one can apply the full expectation formula: {\small{}{}{}{}
\begin{eqnarray*}
u(t+dt,m;z,k) & = & u(t,m+vdt;z,k)(1-\beta dt-\mu dt)+\\
 &  & \beta dt\int_{0}^{1}u(t,\theta m;z,k)\cdot u(t,(1-\theta)m;z,k)q(\theta)d\theta+\mu dt
\end{eqnarray*}
}Theorem 1 is obtained by letting $dt\rightarrow0$. 
\end{proof}
For $k=0$ in equation (\ref{eq:momentgenerating}), it will lead
to the equation for $Ez^{N(t)},$ which we already discussed. Let's
put $z=1$ and study the equation (\ref{eq:basicequations}) as a
function of $k$. Denote $L_{1}(t,m):=E_{m}(M(t))=\frac{-\partial u}{\partial k}|_{z=1,k=0},$
then from equation (\ref{eq:basicequations}), we have 
\begin{equation}
\left\{ \begin{array}{l}
\frac{\partial L_{1}(t,m)}{\partial t}=\frac{\partial L_{1}(t,m)}{\partial m}\ v+2\beta\int_{0}^{1}(L_{1}(t,\theta m)-L_{1}(t,m))q(\theta)d\theta+(\beta-\mu)L_{1}(t,m)\\
\hspace{2.5cm}\\
L_{1}(0,m)=m.
\end{array}\right.\label{eq:mass first moment-2}
\end{equation}

Differentiate the equation (\ref{eq:basicequations}) twice over $k$
and substituting $z=1$ and $k=0$, we will get the equation for the
second moment $L_{2}(t,m):=E_{m}(M^{2}(t))=\frac{\partial^{2}u}{\partial k^{2}}|_{z=1,k=0}$

\begin{equation}
\left\{ \begin{array}{l}
\frac{\partial L_{2}(t,m)}{\partial t}=\frac{\partial L_{2}(t,m)}{\partial m}\ v+2\beta\int_{0}^{1}(L_{2}(t,\theta m)-L_{2}(t,m))q(\theta)d\theta+(\beta-\mu)L_{2}(t,m)\\
+2\beta\int_{0}^{1}(L_{1}(t,\theta m)L_{1}(t,(1-\theta)m))q(\theta)d\theta\\
\hspace{2.5cm}\\
L_{2}(0,m)=m^{2}.
\end{array}\right.\label{eq:mass first moment-1-1}
\end{equation}

The rest of the paper is organized as follows. In section 2, we discuss
the mass process. In section 3, we study the analytic properties of
the limiting mass distribution density. Following this, we discuss
the moment of total mass of the population in section 4. 

\section{KPP model and distribution of mass inside the propogation front of
particles}

One can derive now the equation for the distribution of mass in the
case of random motion of particles (migration).

The classical KPP model describes the evolution of the new particles
in the presence of the branching and random spatial dynamics, say,
diffusion\cite{Kolmogorov_1937}.

We extended the KPP model by considering an extra parameter mass.
We start from the single particle of the mass $m$ located at the
moment $t=0$ in the point $x\in\mathbb{R}^{d}$, i.e $(x,m)\in\mathbb{R}^{d}\times\mathbb{R}_{+}^{1}$.
Evolution of the particle and its mass until the first reaction is
given by the Brownian Motion with the diffusion coefficient $\kappa$ for the space position
	\[ x(t) = x + \kappa b(t)
\]
where $b(t) \in  \mathbb{R}^d$ is a standard Brownian Motion. The generator of $x(t)$ is the usual Laplacian $\mathcal{L} = \frac{\sigma^2 \Delta}{2}$.
. For the mass $m(t)$, we
assume the linear growth: 
\[
m(t)=m+vt,
\]
where $v>0$, though one can study more general processes containing
diffusion.

It starts from the single particle at the point $x\in\mathbb{R}^{d}$
with mass $m$. The particles performs Brownian motion with the generator
$\mathit{\kappa}\Delta$, $\kappa>0$ is the diffusion coefficient.
During time $[t,t+dt]$ each particle in the population can splits
into two particles (offsprings) with probability $\beta dt$, $\beta>0$
is the birth rate. After splitting, this mass is randomly distributed
between offsprings and starts to grow linearly before the next splitting.
During time interval $[t,t+dt]$, any particles of the mass $m$ is
divided into two particles of the random mass $m'=\theta m$, $m''=(1-\theta)m$,
see section 2. Here $\theta\in[0,1]$ is symmetrically distributed
(with respect to the center $0.5\in[0,1]$) random variable with the
density $q(x)=q(1-x)$, $x\in[0,1]$. The offspring perform the same
but independent dynamics like the initial particle.

Like in the standard theory of the reaction-diffusion equations, we
can present the evolution of the particles field as a Markov Process
in the Fock space 
\[
X=\varnothing\cup\underbrace{(\mathbb{R}^{d}\times\mathbb{R}_{+}^{1})}_{\mathcal{F}_{1}}\cup\cdots\underbrace{\left(\mathbb{R}^{d}\times\mathbb{R}_{+}^{1}\right)^{n}}_{\mathcal{F}_{n}}\cup\cdots,
\]
see \cite{Fisher}.

\begin{figure}[H]
\centering{}\includegraphics[width=0.6\columnwidth]{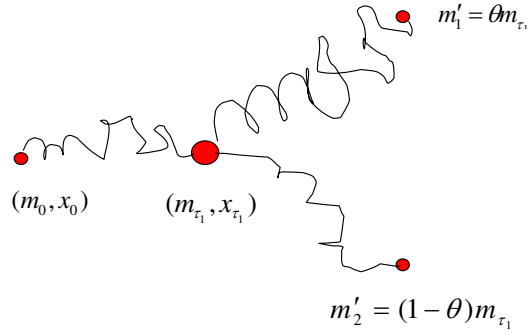}\caption{Evolution of the particles}
\end{figure}

For each open set $\Gamma\subset\mathbb{R}^{d}\times\mathbb{R}_{+}^{1}$,
Let's define the following notations: $N(t,\Gamma):=$ total numbers
of particles at the moment $t$ in the set $\Gamma,$ $M(t,\Gamma):=$
total mass of the particles at the moment $t$ in the set $\Gamma$.
We introduce the generating function of the form: 
\[
u(t,x,m,\Gamma;z,k)=E_{x,m}z^{N(t,\Gamma)}e^{-kM(t,\Gamma)},|z|\leq1,k\geq0,
\]
then we have the following equation for $u(t,x,m,\Gamma;z,k)$. 
\begin{thm}
\label{Thereom moment generating-1}Let $u(t,x,m,\Gamma;z,k)=E_{x,m}z^{N(t,\Gamma)}e^{-kM(t,\Gamma)},$
then $u(t,x,m,\Gamma;z,k)$ satisfy the following functional-differential
equation:

{\small{}{}{} 
\begin{eqnarray}
\frac{\partial u(t,x,m,\Gamma;z,k)}{\partial t} & = & \kappa\Delta u(t,x,m,\Gamma;z,k)-\beta u(t,x,m,\Gamma;z,k)+\frac{\partial u(t,x,m,\Gamma;z,k)}{\partial m}v\nonumber \\
 &  & +\beta\int_{0}^{1}u(t,x,\theta m,\Gamma;z,k)\cdot u(t,x,(1-\theta)m,\Gamma;z,k)q(\theta)d\theta\label{eq:basicequations-1}\\
u(0,m;z,k) & = & ze^{-km}I_{x}(\Gamma).\nonumber 
\end{eqnarray}
}{\small \par}
\end{thm}

\begin{proof}
The formal derivation of this equation is based on the standard technique:
balance of the probabilities in the infinitesimal initial time interval
$[0,dt]$. Namely, let's consider 
\[
u(t+dt,x,m;z,k)=Ex_{x,m}z^{N(t+dt)}e^{-kM(t+dt)}
\]
and then let's split the interval $[0,t+dt]$ into two parts $[0,dt]\cup[dt,t+dt]$.
At the moment $t=0$, we have one particle in the point $x$ with
mass $m$ and during $[0,dt]$, we observe one of the following : 
\begin{itemize}
\item Brownian motion;
\item splitting of the initial particle into two particles with probability
$\beta dt$; 
\item annihilation of the initial particle with probability $\mu dt$; 
\item nothing happen, no annihilation and no splitting with probability
$1-\beta dt-\mu dt.$ 
\end{itemize}
Now one can apply the full expectation formula: {\small{}{}{}{}
\begin{eqnarray*}
u(t+dt,x+\kappa db(t),m;z,k) & = & u(t,x+\kappa db(t),m+vdt;z,k)(1-\beta dt-\mu dt)+\\
 &  & \beta dt\int_{0}^{1}u(t,\theta m;z,k)\cdot u(t,(1-\theta)m;z,k)q(\theta)d\theta+\mu dt
\end{eqnarray*}
}Theorem 1 is obtained by letting $dt\rightarrow0$. 
\end{proof}
The proof of Theorem \ref{Thereom moment generating-1} are practically
identical to the proof of Theorem \ref{Thereom moment generating}.
We will omit the proof here.

Due to the non-linearity, moment generating function is not the best
source of the information about the particle field, it is better to
work with the statistical moments. The factorial moments of $N(t,\Gamma)$
can be calculated by partial derivative of moment generating function
with respect to $z$ at $z=1$. The moments of mass $M(t,\Gamma)$
can be obtained by differentiated with respect to $k$ at $k=0$.

Since $N(t,\Gamma)$ is the number of the particles on the set $\Gamma$,
assume $u_{z}(t,x,\Gamma)=E_{x}z^{N(t,\Gamma)}$ is the generating
function and $x$ is the location of the initial particle, then 
\begin{eqnarray}
\frac{\partial u_{z}}{\partial t} & = & \kappa\Delta u+\beta(u^{2}-u)\label{eq: equationfirst}\\
u_{z}(0,x,\Gamma) & = & z^{I_{\Gamma}}\nonumber 
\end{eqnarray}

Each path along the genealogical tree of the population on the time
interval $[0,t]$ is a Brownian trajectory with the typical range
$O(\sqrt{t})$. The number of the particles is growing exponentially
like $Exp(\beta)$ and due to small large deviation probabilities
the ``radius'' of the population has order $O(t)$. Kolmogorov described
``the boundary or the front'' of the population in terms of the
special solution of the corresponding combustion equation 
\begin{eqnarray}
\frac{\partial v}{\partial t}=\kappa\Delta v+\beta v(1-v)\label{eq: equation second}
\end{eqnarray}

It is equation (\ref{eq: equationfirst}) after substitution $v=1-u$.
In one dimension case, see \cite{Kolmogorov_1937}, the particle soliton
like solution of equation (\ref{eq: equation second}) is given by
\[
v(t,x)=\phi(x-ct)
\]
As $z\to-\infty$, $\phi(z)\to1$ and $z\to\infty$, $\phi(z)\to0$.
The function $\phi(.)$ presents the parameterization of the separatrix
connecting two critical points of the ODE 
\[
\kappa\phi''+c\phi'+\beta\phi(1-\phi)=0.
\]
Such definition of the ``front'' is not the only interesting one.
From the point of view of the population dynamics, another definitions
are also possible. Let $l_{1}(t,x,\Gamma):=E_{x}N(t,\Gamma)$, and
$l_{1}(t,x,\Gamma)=\int_{\Gamma}l_{1}(t,x,y)dy$, function $l_{1}(t,x,y)$
is density of the population at moment t starting from the single
particle at $x\in\mathbb{R}^{d}$. As easy to see, 
\begin{eqnarray*}
\frac{\partial l_{1}(t,x,y)}{\partial t} & = & \kappa\Delta_{x}l_{1}(t,x,y)+\beta l_{1}(t,x,y)\\
l_{1}(0,x,y) & = & \delta_{x}(y)
\end{eqnarray*}
i.e. 
\[
l_{1}(t,0,y)=\frac{\exp{-\frac{y^{2}}{4\kappa t}+\beta t}}{(4\kappa\pi t)^{\frac{d}{2}}}
\]

Define the ``density front'' by the relation $l_{1}(t,0,y)=1$ will
give $|y|\approx2\sqrt{\kappa\beta}t$. This is not a Kolmogorov's
definition of the front, however, it also propagates linearly in time-space.
It is convenient for the moments calculations.

Now let us consider the first moment of $M(t,\Gamma),$ let $L_{1}(t,x,m;\Gamma):=E_{x,m}M(t,\Gamma)$
$=\frac{-\partial u(t,x,m;\Gamma;z,k)}{\partial k}|_{z=1,k=0}$, then

{\small{}{}{} 
\begin{eqnarray}
\frac{\partial L_{1}(t,x,m;\Gamma)}{\partial t} & = & \kappa\Delta L_{1}(t,x,m;\Gamma)+\beta L_{1}(t,x,m;\Gamma)+\frac{\partial L_{1}(t,x,m;\Gamma)}{\partial m}v\nonumber \\
 &  & +2\beta\int_{0}^{1}(L_{1}(t,x,\theta m,\Gamma)-L_{1}(t,x,m,\Gamma))q(\theta)d\theta\label{eq:basicequations-1-1}\\
L_{1}(0,x,m;\Gamma) & = & mI_{x}(\Gamma).\nonumber 
\end{eqnarray}
} From equation (\ref{eq:basicequations-1-1}), we can see that the
operator in the right part without potential term $\beta L_{1}(t,x,m;\Gamma)$
describes two independent Markov processes with the generators: 
\begin{equation}
\mathcal{L}_{x}f=\kappa\Delta f\label{eq:Laplacian operator}
\end{equation}
and 
\begin{equation}
\mathcal{L}_{m}f=v\frac{\partial f}{\partial m}+2\beta\int_{0}^{1}[f(\theta m)-f(m)]q(\theta)d\theta.\label{eq:mass operator-1}
\end{equation}
$\mathcal{L}_{x}$ is the usual Laplacian operator correponding to
Brownina motion, $\mathcal{L}_{m}$ is the generator of mass process
on half axis $m>0$. As a result, one can find the solution of the
first moment 
\begin{eqnarray}
L_{1}(t,x,m;\Gamma)=\int_{\Gamma}\frac{\exp(-\frac{(x-y)^{2}}{4\kappa t}+\beta t)}{(4\kappa\pi t)^{\frac{d}{2}}}\rho(t,m,m')m'dm'dy.\label{eq:firstmoment}
\end{eqnarray}

Consider any bounded open set, say the ball $B_{r}(x)=\{y:|x-y|\leq r\}$,
We have that for any $B_{r}(x)$ inside the front, $E^{2}N(t,B_{r}(x))\ll Variance\ N(t,B_{r}(x))$.
More precisely, one can prove that 
\[
P\{\frac{N(t,B_{r}(x))}{EN(t,B_{r}(x))}>a\}\xrightarrow{t\to\infty}e^{-a}\ \ if\ \frac{|x|}{t}\to0
\]
The limiting distribution is exactly the same like for $\frac{N(t)}{EN(t)}$.
If $t\to\infty,\frac{|x|}{t}\to\gamma<2\sqrt{\kappa\beta}$, the limiting
distribution for $\frac{N(t,B_{r}(x))}{EN(t,B_{r}(x))}$ depends on
$\gamma$! This indicates that particles field in the region $|x|=O(t)$
is more intermittent than in the central zone. It has a structure
of relatively large but sparse clusters. The intermittent structure
of the population inside the propagating front was studied in detail
in the paper \cite{Koralov_2013}.

\section{Mass process}

The equation (\ref{eq:mass first moment-2}) contains the constant
potential $\beta-\mu$ and the operator 
\begin{equation}
\mathcal{L}_{m}f=v\frac{\partial f}{\partial m}+2\beta\int_{0}^{1}\left(f(\theta m)-f(m)\right)q(\theta)d\theta,\label{eq:mass operator}
\end{equation}
which is the generator of the one dimension Markov process $m(t)$.
This mass process $m(t)$ has the following description: it starts
at $t=0$ with the initial mass $m$ and grows linearly $m(t)=m+vt,\,$~$t\leq\tau_{1}$,
where $\tau_{1}$ is exponential distributed random variable with
parameter $2\beta.$ At the moment $\tau_{1}+0$, this particle splits
into two particles with corresponding masses 
\[
m'=(m+v\tau_{1})\theta_{1},
\]
\[
m''=(m+v\tau_{1})(1-\theta_{1}),
\]
where $\theta_{1}$ and $1-\theta_{1}$ has the same density $q(\theta)$.
By definition, 
\[
m(\tau_{1}+0)=(m+v\tau_{1})\theta_{1}.
\]
The graph of $m(t)$ is presented in the following Figure \ref{fig:Mass-process}.

\begin{figure}[H]
\begin{centering}
\includegraphics[width=0.8\columnwidth]{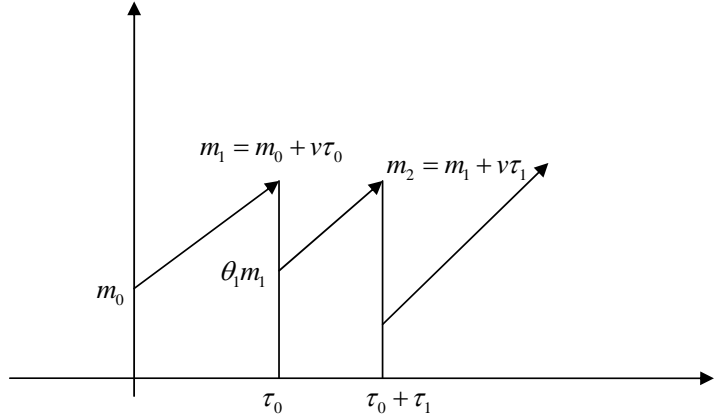} 
\par\end{centering}
\caption{Mass process\label{fig:Mass-process}}
\end{figure}

\textbf{Remark}: Factor $2\beta$ instead of $\beta$ appears due
to the fact that after splitting, we have two identical particles.

Let us consider the embedded chain $m_{n}=(m_{n-1}+v\tau_{n})\theta_{n}$,
$n\geq1$, i.e. the mass process $m(t)$ at the Poisson moments $T_{1}=\tau_{1},$
$T_{2}=\tau_{1}+\tau_{2}$, $\cdots$, $T_{n}=\tau_{1}+\cdots+\tau_{n},$
we will get recursively 
\[
\ m(T_{1})=\theta_{1}(m+v\tau_{1})
\]
Similarly, at the moment of the second splitting, 
\begin{eqnarray*}
m(T_{2}) & = & \theta_{2}(m_{1}+v\tau_{2})\\
 & = & \theta_{1}\theta_{2}m+v\tau_{2}\theta_{2}+v\tau_{1}\theta_{1}\theta_{2}\\
 & \stackrel{law}{=} & \theta_{1}\theta_{2}m+v\tau_{1}\theta_{1}+v\tau_{2}\theta_{1}\theta_{2}
\end{eqnarray*}
In general, 
\[
m(T_{n})\stackrel{law}{=}\theta_{1}\cdots\theta_{n}m+v\tau_{1}\theta_{1}+\cdots+v\tau_{n}\theta_{1}\cdots\theta_{n}
\]
so as $n\to\infty$, the limit will have the the form 
\[
m(T_{n})\xrightarrow[n\rightarrow\infty]{law}m_{\infty}=v\tau_{1}\theta_{1}+\cdots+v\tau_{n}\theta_{1}\cdots\theta_{n}+\cdots
\]
The last random series has all moments since $\tau_{i}$ is exponential
distributed with parameter $2\beta$ and $\theta_{i}$, $i=1,2,\cdots$
are bounded. This chain describes the distribution of the mass of
new born particles at the moments of splitting. Unfortunately, the
law of $m_{\infty}$ is the invariant distribution for the chain $m(T_{n})=m(\tau_{1}+\cdots+\tau_{n})$,
but not for $m(t).$ Let us find the invariant density $\Pi(m)$ for
the process $m(t).$

Denote $\nu(t)$ the number of the Poisson point $T_{i},$ $i=1,2,\cdots$
on the time interval $[0,t],$ i.e, $\nu(t)\sim Poisson(2\beta t)$.
Then for $\nu(t)=n$, 
\begin{eqnarray*}
m(t)=m_{n}+v(t-T_{n})
\end{eqnarray*}
The points $T_{1},\cdots,T_{n}$ divide $[0,t]$ onto $n+1$ sub-interval
(spacing) $\Delta_{1},\cdots,\Delta_{n+1}$ with the same distribution.
They are not independent of course since $\Delta_{1}+\cdots+\Delta_{n+1}=t$.
But the points $T_{1},T_{2},\cdots,T_{n}$ are the ordered statistics
for the set of $n$ independent and uniformly distributed on $[0,t]$
random variable. It is well known \cite{Feller_1971} that the spacing
can be presented in the form 
\begin{eqnarray*}
\Delta_{i}=\frac{Z_{i}t}{Z_{1}+\cdots+Z_{n+1}} & i & =1,\cdots,n+1
\end{eqnarray*}
where $Z_{i}$ are i.i.d random variable with exponential law $Exp(1),$
then for $\nu(t)=n$ 
\begin{eqnarray*}
m(t) & = & \left(\cdots\left(\left(\left(m+\Delta_{1}v\right)\theta_{1}+\Delta_{2}v\right)\theta_{2}+\Delta_{3}v\right)\theta_{3}+\cdots+\Delta_{n}v\right)\theta_{n}+\Delta_{n+1}v\\
 & = & \theta_{1}\cdots\theta_{n}m+\frac{t}{Z_{1}+\cdots+Z_{n+1}}\left(\left(\cdots\left(\left(\left(Z_{1}v\left)\theta_{1}+Z_{2}v\left)\text{\ensuremath{\theta_{2}}+\ensuremath{Z_{3}v}\ensuremath{\left)\theta_{3}+\right.}}\right.\right.\right.\right.\right.\right.\right.\\
 &  & \cdots+\xi_{n}v\left)\theta_{n}+\xi_{n+1}v\left)\right.\right.\\
 & = & \theta_{1}\cdots\theta_{n}m+\frac{t}{Z_{1}+\cdots+Z_{n+1}}\left(\left(\cdots\left(\left(\left(Z_{1}v\left)\theta_{1}+Z_{2}v\left)\text{\ensuremath{\theta_{2}}+\ensuremath{Z_{3}v}\ensuremath{\left)\theta_{3}+\right.}}\right.\right.\right.\right.\right.\right.\right.\\
 &  & \cdots+Z_{\nu(t)}v\left)\theta_{\nu(t)}+Z_{\nu(t)+1}v\left)\right.\right.\\
 & \stackrel{law}{=} & \theta_{1}\cdots\theta_{\nu(t)}m+\frac{t}{\nu(t)}\frac{\nu(t)}{Z_{1}+\cdots+Z_{\nu(t)+1}}\left(\xi_{0}v+\xi_{1}v\theta_{1}+\xi_{2}v\theta_{1}\theta_{2}+\cdots\right)\\
 & \xrightarrow[t\rightarrow\infty]{law} & \frac{v}{2\beta}\left(\xi_{0}+\xi_{1}\theta_{1}+\xi_{2}\theta_{1}\theta_{2}+\cdots\right)
\end{eqnarray*}
where $\xi_{i}$ are standard independent $Exp(1)$ random variable.
Note that in the last step we use the following facts : 
\begin{enumerate}
\item $Z_{i}$ are i.i.d $Exp(1)$ random variable and $\frac{\nu(t)}{Z_{1}+\cdots+Z_{\nu(t)+1}}\xrightarrow[t\rightarrow\infty]{law}\frac{1}{E(Z_{i})}=1;$ 
\item $\nu(t)\sim Poisson(2\beta t)$ and $E(\nu(t))=2\beta t;$ 
\item $\theta_{i}$ are i.i.d random variable. 
\end{enumerate}
We proved the following result: Markov mass process $m(t)$ has the
limiting distribution $\Pi(m)$ which is the law of the random variable
\begin{eqnarray*}
m_{\infty} & = & \frac{v}{2\beta}\left(\xi_{0}+\xi_{1}\theta_{1}+\xi_{2}\theta_{1}\theta_{2}+\cdots\right)\\
 & = & v\left(\tau_{0}+\tau_{1}\theta_{1}+\tau_{2}\theta_{1}\theta_{2}+\cdots\right)
\end{eqnarray*}
where $\tau_{i}\sim Exp(2\beta)$ and $\xi_{i}$ are i.i.d $Exp(1)$
random variable and $\theta_{i}$ are also i.i.d random variable with
the symmetric density $q(x)=q(1-x)$ for $x\in[0,1]$. We will assume
that $q(x)=0$ if $|x-\frac{1}{2}|\geq\delta,$ $0<\delta<\frac{1}{2},i.e.$~$0<\delta\leq\theta_{i}\leq1-\delta,$
$i=1,2,\cdots$

The transition density of the mass process $\rho(t,m,m')$, i.e. the
fundamental solution of 
\begin{eqnarray*}
\left\{ \begin{array}{l}
\frac{\partial\rho(t,m,m')}{\partial t}=\mathcal{L}_{m}\rho(t,m,m')\\
\rho(0,m,m')=\delta_{m'}(m)
\end{array}\right.
\end{eqnarray*}
have a limit $\Pi(m')={\displaystyle \lim_{t\rightarrow\infty}\rho(t,m,m')}$. 
\begin{thm}
Process $m(t)$ has the invariant density $\Pi(m)$ , this density
equals to the distribution density of the random geometric series
\begin{equation}
\xi=v\tau_{0}+v\tau_{1}\theta_{1}+\cdots+v\tau_{n}\theta_{1}\cdots\theta_{n}+\cdots\label{-1}
\end{equation}
Where $\tau_{i},i\geq0$ are i.i.d $Exp(2\beta)$ random variable
and $\theta_{i},i\geq0$ are i.i.d random variable with the probability
density $q(\theta)$, $\tau_{i}$ and $\theta_{i}$ are independent. 
\end{thm}

\textbf{Remark:} 
\begin{itemize}
\item The operator $\mathcal{L}_{m}$ is the unusual Markov generator. It
belongs to the class functional-differential operator with linearly
transformed argument which appear in many applications. See Derfel
et al. \cite{Derfel2009ASFCS}. All such Markov processes are directly
or indirectly related to the solvable group $Aff(R^{1})$ of the transformations
$x\rightarrow ax+b$ of $R^{1}\rightarrow R^{1}.$ This group has
the standard matrix representation $g=\left[\begin{array}{cc}
a & b\\
0 & 1
\end{array}\right],$ $a>0.$The simplest symmetric random walks on this group have the
form 
\begin{eqnarray}
g_{n} & = & \left[\begin{array}{cc}
e^{X_{1}+\cdots+X_{n}} & \sum Y_{i}e^{X_{1}+\cdots+X_{i-1}}\\
0 & 1
\end{array}\right]\label{eq:Affinegroup}
\end{eqnarray}
$\{X_{i}\}$ and $\{Y_{i}\},$ $i\geq1$ are symmetric i.i.d random
vector. The upper of diagonal term has the same structure like $m_{\infty}$.%
\item One can check that the law $m_{\infty}$ is invariant density for
the mass process directly. It is not difficult to verify that $\Pi(m)$
is the solution of the conjugate equation 
\begin{equation}
\mathcal{L}^{\star}g=-v\frac{\partial g}{\partial m}+2\beta\int_{0}^{1}[g(\frac{m}{\theta})-g(m)]q(\theta)d\theta=0\label{eq:rescaling}
\end{equation}
This functional-differential equation with rescaling is similar to
the archetypal equation which was studied in \cite{Derfel1989}, \cite{BDM}
and \cite{BDM_Madrid_2015}. 
\end{itemize}

\section{Analytic properties of the limiting mass distribution density}

Now we'll calculate the moment for the invariant limiting distribution
$\Pi(m)$, the calculation will be based on the following fact: 
\[
\xi=v\tau+\theta_{1}\tilde{\xi}
\]
Here $\tilde{\xi}\stackrel{law}{=}\xi$ and $\tau,\theta_{1}$ are
independent on $\tilde{\xi}$. Since $\tau$ is exponential random
variable with parameter $2\beta$, so 
\[
E\tau=\frac{1}{2\beta},E\tau^{2}=\frac{1}{2\beta^{2}}
\]
As a result, 
\[
E\xi=Ev\tau+E\theta_{1}\tilde{\xi}
\]
so 
\[
E\xi=\frac{v}{2\beta}+\frac{1}{2}E\xi
\]
thus, 
\[
E\xi=\frac{v}{\beta}
\]
The second moment 
\[
E\xi^{2}=v^{2}E(\tau^{2})+2vE(\tau\theta_{1}\tilde{\xi})+E(\theta_{1}^{2}\tilde{\xi})^{2}
\]
from the independence, then 
\[
E\xi^{2}=\frac{v^{2}}{\beta^{2}(1-E\theta_{1}^{2})}
\]
so 
\[
var\xi=\frac{v^{2}}{\beta^{2}}(\frac{1}{1-E\theta_{1}^{2}}-1).
\]
Similarly, the third moments 
\[
E\xi^{3}=\frac{3v^{3}}{2\beta^{3}(1-E\theta_{1}^{2})(1-E\theta_{1}^{3})}.
\]
In general, 
\[
E\xi^{k}=E(v\tau+\theta_{1}\tilde{\xi)}^{k}
\]

i.e.

\[
E\xi^{k}(1-E\theta_{1}^{k})=\sum_{i=0}^{k-1}\left(\begin{array}{c}
k\\
i
\end{array}\right)(vE\tau)^{i}E(\theta\xi)^{k-i}
\]

Let's find the asymptotic of $\Pi(m)$ for large m and small m. Since
\[
\xi=v\tau_{0}+v\tau_{1}\theta_{1}+\cdots+v\tau_{n}\theta_{1}\cdots\theta_{n}+\cdots
\]
Therefore, 
\begin{eqnarray*}
E_{\theta}[e^{-\lambda\xi}] & = & E_{\theta}[e^{-\lambda(v\tau_{0}+v\tau_{1}\theta_{1}+v\tau_{2}\theta_{1}\theta_{2}\cdots)}]\\
 & = & E_{\theta}[e^{-\lambda v\tau_{0}}]E_{\theta}[e^{-\lambda v\tau_{1}\theta_{1}}]\cdots[e^{-\lambda v\tau_{2}\theta_{1}\cdots\theta_{n}}]\cdots\\
 & = & \frac{1}{(1+\frac{\lambda v}{2\beta})(1+\frac{\lambda v\theta_{1}}{2\beta})\cdots(1+\frac{\lambda v\theta_{1}\cdots\theta_{n}}{\xi2\beta})\cdots}\\
 & = & \frac{c_{0}}{1+\frac{\lambda v}{2\beta}}+\frac{c_{1}}{1+\frac{\lambda v\theta_{1}}{2\beta}}+\cdots+\frac{c_{n}}{1+\frac{\lambda v\theta_{1}\cdots\theta_{n}}{2\beta}}+\cdots\\
 & = & \frac{c_{0}}{1+\frac{\lambda}{\frac{2\beta}{v}}}+\frac{c_{1}}{1+\frac{\lambda}{\frac{2\beta}{v\theta_{1}}}}+\cdots+\frac{c_{n}}{1+\frac{\lambda}{\frac{2\beta}{v\theta_{1}\cdots\theta_{n}}}}+\cdots
\end{eqnarray*}
Here, 
\[
c_{0}=\frac{1}{(1-\theta_{1})(1-\theta_{1}\theta_{2})\cdots}
\]
\[
c_{1}=\frac{1}{(1-\frac{1}{\theta_{1}})(1-\theta_{2})(1-\theta_{2}\theta_{3})\cdots}
\]
\[
c_{n}=\frac{1}{(1-\frac{1}{\theta_{1}\cdots\theta_{n}})(1-\frac{1}{\theta_{2}\cdots\theta_{n}})\cdots(1-\frac{1}{\theta_{n}})(1-\theta_{n+1})(1-\theta_{n+1}\theta_{n+2})\cdots}
\]
Now one can find conditional density $p_{\xi}(m)$ of random variable
if $\vec{\theta}=(\theta_{1},\theta_{2},\cdots)$ are known, 
\begin{eqnarray}
p_{\xi}(m) & = & E_{\vec{\theta}}[\frac{2\beta}{v}e^{-\frac{2\beta m}{v}}c_{0}]+E_{\vec{\theta}}[\frac{2\beta}{v\theta_{1}}e^{-\frac{2\beta m}{v\theta_{1}}}c_{1}]+\cdots\nonumber \\
 & = & E_{\vec{\theta}}(\frac{2\beta}{v}e^{-\frac{2\beta m}{v}}\frac{1}{(1-\theta_{1})(1-\theta_{1}\theta_{2})\cdots})\nonumber \\
 & + & E_{\vec{\theta}}(\frac{2\beta}{v\theta_{1}}e^{-\frac{2\beta m}{v\theta_{1}}}\frac{1}{(1-\frac{1}{\theta_{1}})(1-\theta_{2})(1-\theta_{2}\theta_{3})\cdots})+\cdots\nonumber \\
 & = & \frac{2\alpha\beta}{v}e^{-\frac{2\beta m}{v}}-E_{\theta_{1}}\frac{2\alpha\beta}{1-\theta_{1}}e^{-\frac{2\beta m}{v\theta_{1}}}+\cdots\label{chap1equation5-1}
\end{eqnarray}
where 
\begin{equation}
\alpha=E\frac{1}{(1-\theta_{1})(1-\theta_{1}\theta_{2})\cdots}\label{chap1equation6-1}
\end{equation}

Let's formulate several analytic results about the invariant density
. 
\begin{thm}
Assume that $Supp\theta=[a,1-a]$, $0<a\leq\frac{1}{2}$, then for
large m, 
\[
\Pi(m)\xrightarrow[m\rightarrow\infty]{}\frac{2\alpha\beta}{v}e^{-\frac{2\beta m}{v}}+R(m)
\]
The remainder term with the maximum on the boundary has order 
\[
R(m)\sim\frac{2\alpha\beta}{a}e^{-\frac{2\beta m}{v(1-a)}}L(m)
\]
Where $L(m)\xrightarrow[m\to\infty]{}0$ and $L(m)$ depends on the
structure of the distribution $q(d\theta)$ near the maximum point
$\theta_{critical}=1-a.$ 
\end{thm}

\begin{proof}
From (\ref{chap1equation5-1}), due to the Laplace method, it is trivial
to get the result. 
\end{proof}
The behavior of $p_{\xi}(m)$ as $m\to0$ is much more interesting.
Here we will use the Exponential Chebyshev's inequality. More detailed
analysis in the case when $q(d\theta)$ is a discrete (atomic) measure,
has been done in Derfel \cite{Derfel1978}, Cooke \& Derfel \cite{Cooke_Derfel_1996}.

For instance, the following result is true for the pantograph equation
(\ref{eq:pantograph}).

(i) Steady- state solution of (\ref{eq:pantograph}) satisfies the
following estimate

\begin{equation}
|y(x)|<D\exp\{{-b\ln^{2}|x|}\};\qquad D>0,\qquad b=\frac{1}{2\ln\alpha}
\end{equation}

in some neighborhood of zero.

(ii) On the other hand, every solution of (\ref{eq:pantograph}) which
satisfies estimate $y(x)|<D\exp\{{-a\ln^{2}|x|}\}$ for with some
$a>b$ is identically equal zero.

Similar results are valid also for more general equation 
\begin{equation}
y(x)=\sum_{j=0}^{l}\sum_{k=0}^{n}a_{jk}y^{(k)}(\lambda_{j}x),
\end{equation}
where $\lambda_{j}\neq0$ under the assumption that $\Lambda={\max|\lambda_{j}|}<1$.
Namely, statements (i) and (ii) are fulfilled with $b=\frac{1}{2|\ln\lambda|}$
and $a>\frac{m|\ln\lambda|}{2\ln^{2}\Lambda}$, where $\lambda={\min|\lambda_{j}|}$.

We conjecture that similar asymptotic behavior occurs also for our
model, but currently can prove the following weaker result, only.
The asymptotic approximation is shown in Figure \ref{fig:Asymptotic-behavior-of}. 
\begin{thm}
Assume that $Supp\theta=[a,1-a]$, $0<a\leq\frac{1}{2}$, then if
$m\rightarrow0$, then 
\[
P\{\xi\leq m\}\leq e^{-c_{1}\ln^{2}(\frac{1}{m})}
\]
where $c_{1}$is some constant. 
\end{thm}

\begin{proof}
Let's start from the standard calculations, for $\lambda>0$ and fix
$a\leq\theta_{i}\leq1-a$, $i=1,2,\cdots$ 
\begin{eqnarray}
P\left\{ \xi\leq m|\vec{\theta}\right\}  & = & P\{e^{-\lambda\xi}>e^{-\lambda m}|\vec{\theta}\}\leq\underset{\lambda>0}{\min}\frac{Ee^{-\lambda\xi}}{e^{-\lambda m}}\nonumber \\
 & = & \underset{\lambda>0}{\min}\ e^{\lambda m-\ln(1+\frac{\lambda v}{2\beta})-\ln(1+\frac{\lambda v\theta_{1}}{2\beta})-\ln(1+\frac{\lambda v\theta_{1}\theta_{2}}{2\beta})-\cdots}\label{chap1equation7-1}
\end{eqnarray}
Equation for the critical point $\lambda_{0}=\lambda_{0}(m)$ has
a form: 
\[
m=\frac{\frac{v}{2\beta}}{1+\frac{\lambda v}{2\beta}}+\frac{\frac{v\theta_{1}}{2\beta}}{1+\frac{\lambda v\theta_{1}}{2\beta}}+\cdots+\frac{\frac{v\theta_{1}\cdots\theta_{k}}{2\beta}}{1+\frac{\lambda v\theta_{1}\theta_{2}}{2\beta}}+\cdots
\]
i.e. 
\[
m=\frac{1}{\frac{2\beta}{v}+\lambda}+\frac{1}{\frac{2\beta}{v\theta_{1}}+\lambda}+\cdots+\frac{1}{\frac{2\beta}{v\theta_{1}\cdots\theta_{k}}+\lambda}+\cdots
\]
Define $k(\lambda)=\min\{k:\frac{2\beta}{v\theta_{1}\cdots\theta_{k}}\sim\lambda\}$,
then $m\sim\frac{k(\lambda)}{\lambda}$. From $\frac{2\beta}{v\theta_{1}\cdots\theta_{k}}\sim k$,
we then have $k(\lambda)\sim\frac{\ln{\lambda}}{E\ln(\frac{1}{\theta})}$.
Hence, the critical point 
\begin{eqnarray}
\lambda\sim\frac{\ln(\frac{1}{m})}{mE\ln(\frac{1}{\theta})}\label{chap1equation8-1}
\end{eqnarray}
Substitute (\ref{chap1equation8-1}) into Chebyshev's inequality (\ref{chap1equation7-1})
gives 
\begin{eqnarray*}
P\left\{ \xi\leq m|\vec{\theta}\right\}  & \leq & e^{\frac{\ln(\frac{1}{m})}{E\ln(\frac{1}{\theta})}-\ln\left(1+\frac{v\ln(\frac{1}{m})}{2m\beta E\ln(\frac{1}{\theta})}\right)-\cdots-\ln\left(1+\frac{v\ln(\frac{1}{m})\theta_{1}\cdots\theta_{k}}{2m\beta E\ln(\frac{1}{\theta})}\right)}\\
 & \leq & e^{\frac{\ln(\frac{1}{m})}{E\ln(\frac{1}{\theta})}-\ln\left(1+\frac{v\ln(\frac{1}{m})}{2m\beta E\ln(\frac{1}{\theta})}\right)-\cdots-\ln\left(1+\frac{v\ln(\frac{1}{m})a^{k}}{2m\beta E\ln(\frac{1}{\theta})}\right)}\\
 & \leq & e^{\frac{\ln(\frac{1}{m})}{E\ln(\frac{1}{\theta})}-\sum_{i=0}^{k}\ln\left(\frac{v\ln(\frac{1}{m})a^{i}}{2m\beta E\ln(\frac{1}{\theta})}\right)}\\
 & \leq & e^{-c_{1}\ln^{2}\left(\frac{1}{m}\right)}
\end{eqnarray*}
\end{proof}
\begin{figure}[H]
\begin{centering}
\includegraphics[width=0.99\columnwidth,height=0.45\textheight]{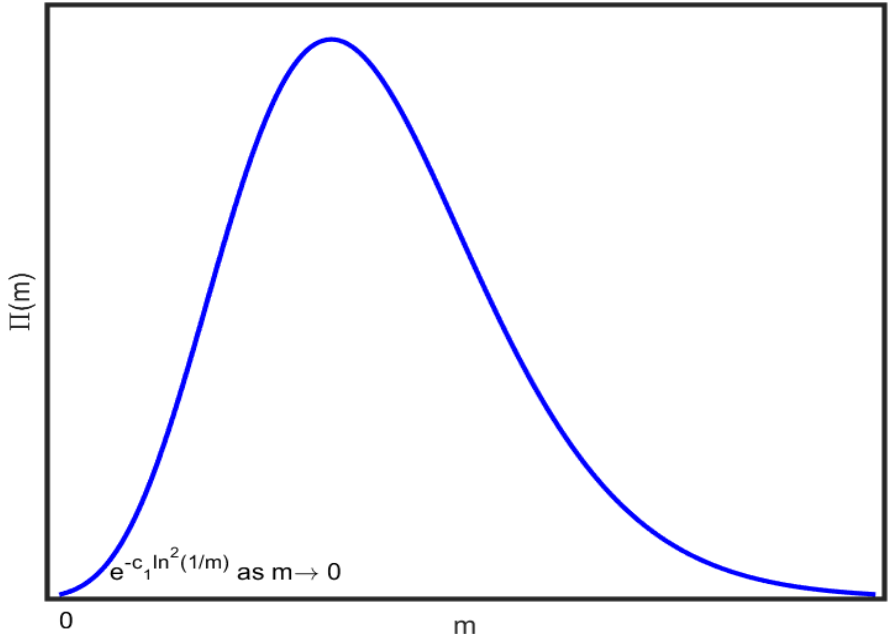} 
\par\end{centering}
\caption{\label{fig:Asymptotic-behavior-of}Asymptotic behavior of $\Pi(m)$}
\end{figure}

\section{Moments of total mass of population $M(t)$}

In this section, we will study the first moment and second moment
of the total mass of population $M(t)$. As discussed in section 1,
the first moment $L_{1}(t,m)$ is given by equation 
\begin{equation}
\left\{ \begin{array}{l}
\frac{\partial L_{1}(t,m)}{\partial t}=\frac{\partial L_{1}(t,m)}{\partial m}\ v+2\beta\int_{0}^{1}(L_{1}(t,\theta m)-L_{1}(t,m))q(\theta)d\theta+(\beta-\mu)L_{1}(t,m)\\
\hspace{2.5cm}\\
L_{1}(0,m)=m
\end{array}\right.\label{eq:mass first moment}
\end{equation}
\begin{cor}
Let $L_{1}(t,m)=E_{m}(M(t))$ then for $t\rightarrow\infty,$ 
\begin{eqnarray*}
L_{1}(t,m)\rightarrow e^{(\beta-\mu)t}\frac{v}{\beta}
\end{eqnarray*}
\end{cor}

\begin{proof}
From equation (\ref{eq:mass first moment}), Duhamel's formula gives
us

\begin{eqnarray*}
L_{1}(t,m)=e^{(\beta-\mu)t}\int_{0}^{\infty}\rho(t,m,m')m'dm'
\end{eqnarray*}
as $t\rightarrow\infty,$ 
\begin{eqnarray*}
L_{1}(t,m)\rightarrow e^{(\beta-\mu)t}\int_{0}^{\infty}\Pi(m')m'dm'=e^{(\beta-\mu)t}\frac{v}{\beta}.
\end{eqnarray*}
The last equality use both Theorem 2 $\rho(t,m,m')\rightarrow\Pi(m')$
and the fact that $E\xi=\frac{v}{\beta}.$ 
\end{proof}
The second moment $L_{2}(t,m)=E_{m}(M(t)^{2})$ is given by 
\begin{equation}
\left\{ \begin{array}{l}
\frac{\partial L_{2}(t,m)}{\partial t}=\frac{\partial L_{2}(t,m)}{\partial m}\ v+2\beta\int_{0}^{1}(L_{2}(t,\theta m)-L_{2}(t,m))q(\theta)d\theta+(\beta-\mu)L_{2}(t,m)\\
+2\beta\int_{0}^{1}(L_{1}(t,\theta m)L_{1}(t,(1-\theta)m))q(\theta)d\theta\\
\hspace{2.5cm}\\
L_{2}(0,m)=m^{2}
\end{array}\right.\label{eq:mass first moment-1}
\end{equation}
From equation (\ref{eq:mass first moment-1}), we have $\frac{\partial L_{2}}{\partial t}=\mathcal{L}_{m}L_{2}+f(t,m)$,
here $\mathcal{L}_{m}$ is the operator of the mass process. By applying
Duhamel's principle and one can find that

\begin{eqnarray*}
L_{2}(t,m)=2\left(e^{(\beta-\mu)t}\frac{v}{\beta}\right)^{2}+O(e^{(\beta-\mu)t})
\end{eqnarray*}

\end{document}